\documentclass[11pt]{article}
\usepackage{epsfig}
\usepackage{graphicx}
\usepackage{amssymb,amsmath}
\usepackage{caption}
\usepackage{cancel}
\usepackage{tikz}
\newtheorem{thm}{{\bf Theorem}}[section]
\newtheorem{theorem}{{\bf Theorem}}[section]

\newtheorem{pro}[thm]{\bf Proposition}

\newtheorem{preproof}{{\bf Proof.}}

\newenvironment{proof}[1]{\begin{preproof}{\rm
               #1}\hfill{$\rule{2mm}{2mm}$}}{\end{preproof}}
\begin{document}
\title{\Large {\bf Metric dimension of Andr\'asfai graphs}}

\author{
{\sc S. Batool Pejman\thanks{b.pejman@edu.ikiu.ac.ir}} , {\sc Shiroyeh Payrovi\thanks{shpayrovi@ikiu.ac.ir}} , {\sc Ali Behtoei\thanks{Corresponding author, a.behtoei@sci.ikiu.ac.ir}}  \\
[1mm]
{\it \small Department of Mathematics, Imam Khomeini International University, 34149-16818, Qazvin, Iran }
}
\date{}

\maketitle

\begin{abstract}
 A set $W\subseteq V(G)$ is called a  resolving set, if
for each pair of distinct vertices $u,v\in V(G)$ there exists $t\in W$
such that $d(u,t)\neq d(v,t)$,  where $d(x,y)$ is the distance
between  vertices $x$ and $y$. The cardinality of a minimum
resolving set for $G$ is called the  metric dimension of $G$ and
is denoted by $\dim_M(G)$.  This parameter has many applications in different areas. The problem of finding metric dimension is NP-complete for general graphs but it is determined for trees and some other important families of graphs.
In this paper, we determine the exact value of the metric dimension of $Andr\acute{a}sfai$ graphs, their complements and $And(k)\square P_n$. 
Also, we provide  upper and lower bounds for $dim_M(And(k)\Box C_n)$.
\end{abstract}

\\
{\bf MSC code 2010: } 05C12, 05C25. 

\section{Introduction}

Throughout this paper all graphs are finite, simple and undirected. 
Let $G=(V,E)$ be a  connected graph with vertex set $V$ and edge set $E$. 
The distance between two vertices $x,y\in V$ is the length of a shortest path between them and is denoted by $d_G(x,y)$, or $d(x,y)$ for convenient. 
The neighborhood of $x$ is $N(x)=\{y\in V :~d(x,y)=1 \}$   and the diameter of $G$ is  $diam(G)=\max\{ d(x,y):~x,y\in V\}$. 
It is well known that almost all graphs have diameter 2.
The notations $\overline{G}$ and $Line(G)$ stand for the complement graph and  the line graph of $G$, respectively.
For an ordered subset $W=\{w_1,w_2,\ldots,w_k\}$ of vertices and a
vertex $v\in V$, the $k$-vector
$r(v|W):=(d(v,w_1),d(v,w_2),\ldots,d(v,w_k))$ is  called  the
\textit{metric representation} of $v$ with respect to $W$ (the \textit{code} of $v$, for convenient). The set $W$ is
called  a \textit{resolving set} for $G$ if distinct vertices of $G$ have
distinct metric representations with respect to $W$.
The cardinality of a  minimum resolving set is  the \textit{metric dimension} of $G$ and is denoted by $\dim_M(G)$.
 A graph with metric dimension $k$ is called $k$-\textit{dimensional}.
These concepts were introduced by  Slater in 1975 when he was working with U.S. Sonar and Coast Guard Loran stations and he described the usefulness of these concepts, \cite{Slater1975}.
Independently, Harary and Melter \cite{Harary} discovered these
concepts.
They have applications in many areas
including network discovery and verification~\cite{net2}, robot navigation~\cite{landmarks}, problems of pattern recognition and image processing~\cite{digital},
coin weighing problems~\cite{coin},
strategies for the Mastermind game~\cite{Mastermind},
combinatorial search and optimization~\cite{coin}.
Determining the metric dimension of different families of graphs, operations and products, or characterizing $n$-vertex graphs with a specified metric dimension are fascinating problems and atracts the attention of many researchers.
The problem of finding metric dimension is NP-Complete for general graphs  but the metric dimension of trees can be  obtained  using a polynomial time algorithm \cite{landmarks}.
 It is not hard to see that for each $n$-vertex graph $G$ we have $1\leq \dim_M(G) \leq  n-1$.
Khuller $et al.$~\cite{landmarks} and Chartrand et al.~\cite{Ollerman}
  proved that $\dim_M(G)=1$ if and only if $G$ is a path $P_n$.
 Chartrand et al.~\cite{Ollerman} proved that for $n\geq 2$, $\dim_M(G)=n-1$
if and only if $G$ is the complete graph $K_n$. The metric dimension of each complete $t$-partite graph with $n$ vertices is $n-t$. They also provided a
 characterization of  graphs of order $n$ with
metric dimension $n-2$, see \cite{Ollerman}.  Graphs of order $n$  with metric dimension $n-3$ are characterized in~\cite{n-3}.
B\'ela Bollob\'as studied the metric dimension of random graphs \cite{Bollobas}.
C$\acute{a}$ceres $et~al.$ studied  this parameter for the Cartesian product of graphs \cite{cartesian product}.
Bailey and Cameron \cite{Cameron} have computed the exact value of the metric dimension for the diameter 2 Kneser and Johnson graphs.
Fijav$\check{z}$ and Mohar studied this parameter for Paley graphs  \cite{Mohar}.
Salman $et~al.$ studied this parameter for the Cayley graphs on cyclic groups \cite{2}.
In  \cite{CayleyDiGraph} and \cite{Fehr} the metric dimension of Cayley digraphs for the groups which are direct product of some cyclic groups is investigated.
Imran studied the metric dimension of barycentric subdivision of Cayley graphs in \cite{Imran}.
  Each cycle graph $C_n$ is a  $2$-dimensional graph . 
  In~\cite{chang}  some properties of $2$-dimensional graphs are obtained.
  All of  2-trees with metric dimension two are characterized in \cite{2-tree}, $2$-dimensional Cayley graphs on Abelian groups are characterized in \cite{CayAbel} and   $2$-dimensional Cayley graphs on dihedral groups are characterized in \cite{Dihedral}.
For more results in this subject or related subject see  \cite{?}.

Recall that the {\em Cartesian product} of two graphs $G_1$ and $G_2$, denoted by $G_1\Box G_2$, is a graph with vertex set $V(G_1)\times V(G_2):=\{(u,v): u\in V(G_1), v\in V(G_2)\}$, in which $(u,v)$ is adjacent to $(u',v')$ whenever $u=u'$ and $vv'\in E(G_2)$, or $v=v'$ and $uu'\in E(G_1)$. 
Note that the vertex set of $G_1\Box G_2$ can be arranged in $|V(G_2)|$ rows and $|V(G_1)|$ columns.
Also if $G_1$ and $G_2$ are connected, then $G_1\Box G_2$ is connected.
Let $H$ be a group and let $S$ be a subset of $H$ that is closed under taking inverse and does not contain the identity element.
Recall that the {\it Cayley graph} $Cay(H,S)$ is a simple graph whose vertex set is $H$ and two vertices $u$ and $v$ are adjacent in it when $uv^{-1}\in S$.
For any integer $k\geq 1$, the  Andr\'asfai graph $And(k)$ is the Cayley graph $Cay(\mathbb{Z}_{3k-1},S)$ where $\mathbb{Z}_{3k-1} =\{1,2, ..., 3k-1=0\}$ is the additive group  of integers modulo $3k-1$ and $S=\{1,4,7,..., 3k-2  \}$ is the subset of $\mathbb{Z}_{3k-1}$ consisting of the elements congruent to $1$ modulo $3$.  Note that  $And(1)$  is a path with two vertices, $And(2)$ is isomorphic to the $5$-cycle and $And(3)$ is a $M\ddot{o}bius~ladder$. 
It is well known that $And(k)$ is a reduced (twin free), circulant, vertex transitive,  triangle-free and  $k$-regular graph whose diameter is two for $k\geq 2$.



In this paper, we determine the exact value of the metric dimension of $Andr\acute{a}sfai$ graphs, their complements and $And(k)\square P_n$. 
Also, we prove that $k\leq dim_M(And(k)\Box C_n)\leq k+1$.

\section{Main Results}
Note that if $W$ is a resolving set for $G$, then for each $v\in V\setminus W$ the set $W\cup \{v\}$  is a larger resolving set for $G$. 
Also, when $G$ is a graph with diameter 2, then $W\subseteq V$ is a resolving set for $G$ if and only if  for each pair of distinct vertices $u,v\in V\setminus W$ there exist $w\in W$ such that $\{d(w,u),d(w,v)\}=\{1,2\}$. 

\begin{theorem} \label{And(k)}
Let $k\geq 1$ be an integer. Then, the metric dimension of the Andr\'asfai graph $And(k)$ is $k$.
\end{theorem}
\begin{proof}{
By investigation, it is easy to see that $\dim_M(And(k))=k$ for $k\in\{1,2,3\}$. Hence, we assume that $k\geq 4$. 
Note that $And(k)=Cay(\Bbb{Z}_{3k-1}, S)$ where  $S=\{1,4,7, ...,3k-2\}$. Since $1\in S$, $And(k)$ contains the Hamiltonian cycle $1,2,3,...,3k-1$ and we can consider a drawing of it in such a way that vertices  are  consecutively ordered clockwise around a cycle.
Hereafter, all of vertex numbers will be cosidered in modulo $3k-1$.
It is straightforward to check that the vertex $0=3k-1$ is adjacent to every vertex in $S$ and each vertex $x\neq 0$ has at least one non-adjacent vertex in $S$.
Consider the subset  $\{t,t+3\}$ of $S$ with $1\leq t \leq 3k-5$. 
We have $d(t+1,t)=1=d(t+2,t+3)$ and $d(t+1,t+3)=2=d(t+2,t)$. 
Also,  for each vertex $y\notin \{t,t+1,t+2,t+3\}$ we have $d(y,t)=1$ if and only if $d(y,t+3)=1$ (because in modulo $3k-1$, we have $t-y\in S$ if and only if $ t+3-y\in S$).
Therefore, $r(t+1|S)\neq r(t+2|S)$, $r(y|S)\neq r(t+1|S)$ and $r(y|S)\neq r(t+2|S)$.
This means that  two vertices $t+1$ and $t+2$ have unique codes among the vertices of $And(k)$. 
Since for each vertex $0\neq x\notin S$ there exists $1\leq t \leq 3k-5$ such that $x=t+1$ or $x=t+2$, the code of $x$ is unique. Hence, $S$ is a resolving set for $And(k)$ and this implies that $\dim_M(And(k))\leq |S|=k$. 
\\
In order to complete the proof, it is sufficient to show that $|W|\geq k$ for each resolving set $W$ of $And(k)$.
Suppose on the contrary that there exists a resolving set $W$ of $And(k)$ with $|W|< k$. 
By including some additional vertices to $W$ (if it is necessary) we can assume that $|W|=k-1$.
If there exists a subset of four (clockwise) consecutive vertices $T=\{ i,i+1,i+2,i+3 \}$ such that $T \cap W=\emptyset$, then for each vertex $j\notin T$ we have $d(j,i)=1$ if and only if $d(j,i+3)=1$ (because, $i-j\in S$ if and only if $i+3-j\in S$). 
This implies that two vertices $i$ and $i+3$ have the same metric representations with respect to $W$, which contradics the resolvability of $W$. 
Now, assume that $W=\{ i_1,i_2,...,i_{k-1}  \}$ where 
$$1\leq i_1<i_2<\cdots < i_{k-1}\leq 3k-1.$$
For each $i_j\in W$ (and with the assumption that $i_k=i_1$) let 
$$B_{i_j}=\{ i_j,i_j+1,i_j+2,...,i_{j+1}  \}\setminus \{i_j,i_{j+1}\}.$$
Note that  $B_{i_j}=\emptyset$ just  when $i_j+1=i_{j+1}$ and that $B_{i_j}\cap W=\emptyset$ for each $i_j\in W$. 
Also, using previous facts we have
$$\bigcup_{j=1}^{k-1}B_{i_j}=\Bbb{Z}_{3k-1} \setminus W,~~ \{ |B_{i_j}|:~ i_j\in W\} \subseteq \{0,1,2,3\}.$$
For each $s\in \{0,1,2,3\}$ let $\beta_s$ be the number of blocks $B_{i_j}$ with $|B_{i_j}|=s$.  
Thus, using the fact $|W|=k-1$ we see that
$$\beta_0+\beta_1+\beta_2+\beta_3=k-1$$
and 
$$ 0\beta_0+1\beta_1+2\beta_2+3\beta_3 = | \bigcup_{j=1}^{k-1}B_{i_j} | =(3k-1)-(k-1)=2k.$$
Therefore,
\begin{eqnarray*}
-2\beta_0-\beta_1+\beta_3 &=& (0\beta_0+1\beta_1+2\beta_2+3\beta_3 )-2(\beta_0+\beta_1+\beta_2+\beta_3) \\
& = & (2k)-2(k-1) \\
& = & 2.
\end{eqnarray*}
This implies that $\beta_3=2+2\beta_0+\beta_1\geq 2$. Specially, $\beta_3 > \beta_0+\beta_1$. 
Now the Pigeonhole Principle implies that  there  exist two blocks $B_{i_j},B_{i_{j'}}$ of size 3 such that between them in at least one direction (clockwise or counterclockwise) only blocks of size 2  (if any exists) are located. 
Since  $And(k)$ is vertex transitive,  whithout loss of generality and for convenient, we can assume that $i_j=1$ (i.e $B_{i_j}=B_1$) and  $B_{i_{j'}}$ is located  (in clockwise direction) after $\ell\geq 0$ blocks   of size two (when $\ell\geq 1$ they are  $B_{5},B_8,...,B_{3\ell+2}$), i.e $B_{i_{j'}}=B_{3\ell+5}$. 
Note that for the case $\ell=0$ two blocks $B_{i_{j}}=B_1$ and  $B_{i_{j'}}=B_5$ are consecutive.
Therefore, 
$$W\cap \{1,2,3,...,3\ell+9 \}=\{1,5,...,3\ell+5, 3\ell+9 \}.$$ 
Now consider two vertices $x=2\in B_1$ and $y=3\ell+8\in B_{3\ell+5}$.
Since $y=x+3(\ell+2)$, for each $z\notin \{1,2,3,...,3\ell+9 \}$ we have $d(z,x)=1$ if and only if $d(z,y)=1$. 
Also, it is straightforward to check that
 $$N(x)\cap \{1,5,...,3\ell+5, 3\ell+9 \} = \{1,3\ell+9\}= N(y)\cap \{1,5,...,3\ell+5, 3\ell+9 \}.$$
This means that  $r(x|W)=r(y|W)$, which is a contradiction. Therefore, $|W|\geq k$ and this completes the proof.
}\end{proof}


Note that $And(1)$ is a 2-vertex path and hence, its complement $\overline{And(1)}$ is disconnected. 
$\overline{And(2)}$ is a 5-cycle and its metric dimension is 2.
Also, for each $k\geq 3$ the complement  of $And(k)$ is a connected $(2k-2)$-regular graph and its diameter is two.
\begin{theorem}
For each $k\geq 2$ we have $\dim_M(\overline{And(k)})=k$.
\end{theorem}
\begin{proof}{
Let $W$ be a  non-empty ordered subset of  $\Bbb{Z}_{3k-1}$ and let $v\in \Bbb{Z}_{3k-1}$ be an arbitrary vertex. 
Assume that  the metric representation of vertex $v$ with respect to $W$ in  $And(k)$ is $r(v|W)=(w_1,w_2,...,w_k)$  and  the metric representation of  $v$ with respect to $W$ in  $\overline{And(k)}$ is $\bar{r}(v|W)=(\bar{w_1},\bar{w}_2,...,\bar{w}_k)$. 
Since both graphs $And(k)$ and  $\overline{And(k)}$ have diameter two,  for each $i\in\{1,2,...,|W|\}$ we have
$$\bar{w}_i=\left\{ \begin{array}{ll}  0 & w_i=0 \\ 1 & w_i=2 \\ 2 & w_i=1.    \end{array}\right.$$
This means that for each vertex $u$ we have $r(v|W)=r(u|W)$ if and only if $\bar{r}(v|W)=\bar{r}(u|W)$.
Thus, there is a one-to-one correspondence between the vectors $\{r(v|W)~|~v\in V(And(k)\}$ and the vectors $\{\bar{r}(v|W)~|~v\in V(\overline{And(k)}\}$ by a switching on non-zero components.
In the proof of Theorem \ref{And(k)} we see that $S$ is a minimum resolving set for $And(k)$ and hence, $|\{r(v|S)~|~v\in V(And(k)\}|=|V(And(k)|$. 
Therefore,  $S$ is a minimum resolving set for $\overline{And(k)}$  and the result follows. 
}\end{proof}

In the following theorem we determine $\dim_M(And(k)\square P_n)$ and $\dim_M(\overline{And(k)}\square P_n)$.
\begin{theorem}
For each $k\geq 1$ and $n\geq 2$ we have 
$$\dim_M(And(k)\square P_n)=\dim_M(\overline{And(k)}\square P_n)=k.$$
Specially, the metric dimension of the prism generated by $And(k)$ or its complement is $k$.
\end{theorem}
\begin{proof}{
Assume that  $V(P_n)=\{v_1,v_2,...,v_n\}$ and $E(P_n)=\{v_1v_2,v_2v_3,...,v_{n-1}v_n \}$. Hence,
$$V(And(k)\square P_n)=\bigcup_{t=1}^n  \{(1,v_t),(2,v_t),...,(3k-1,v_t )\}$$
and the induced subgraph of $And(k)\square P_n$ on the set $\{(1,v_t),(2,v_t),...,(3k-1,v_t)\}$ is isomorphic to $And(k)$ for each $t\in\{1,2,...,n\}$. 
Using Corollary 3.2 in \cite{cartesian product} and Theorem \ref{And(k)} we see that 
$$k=\max\{ \dim_M(And(k)), \dim_M(P_n)\} \leq \dim_M(And(k)\square P_n).$$
Let $$W=\{ (1,v_1),(4,v_1),(7,v_1),....,(3k-2,v_1)  \}=S\times \{v_1\}.$$
 We want to show that $W$ is a resolving set for $And(k)\square P_n$.
For each $i,j\in \Bbb{Z}_{3k-1}$ and for each $t,t'\in\{1,2,...,n\}$ it is easy to see that 
$$d_{And(k) \square P_n} ((i,v_t),(j,v_{t'}))=d_{And(k)}(i,j)+|t-t'|,$$
which implies that
$$ r((i,v_t)|W)=r(i|S)+(t-1,t-1,...,t-1). $$
Note that except the vertex $(0,v_1)$ whose metric representation with respect to $W$ is the all 1 vector $(1,1,...,1)$,
the metric representation of each vertex $(i,v_1)$ has at least one component equal to 2 and  we have $r((i,v_1)|W)\in \{0,1,2\}^k$. 
Similarly, for each $t\in\{2,3,...,n\}$, except the vertex $(0,v_t)$ whose metric representation  is $(t,t,...t)$, the metric representation of each vertex $(i,v_t)$ has at least one component equal to $t+1$ and  $r((i,v_t)|W)\in \{t-1,t,t+1\}^k$. 
By the proof of Theorem \ref{And(k)}, $S$ is a minimum resolving set for $And(k)$. 
Note that by Lemma 3.1 in \cite{cartesian product} the projection of $W$ onto each copy of $And(k)$ in $And(k)\square P_n$ (i.e the induced
subgraph on each row) resolves the vertices of that copy (row).
Therefore, each pair of distinct vertices $(i,v_t)$ and $(j,v_{t'})$ (with $t=t'$ or $t\neq t'$) have distinct metric representations with respect to $W$. Hence, $W$ is a minimum resolving set for  $And(k)\square P_n$ and $\dim_M(And(k)\square P_n)=k$. Using a similar argument we can show that $\dim_M(\overline{And(k)}\square P_n)=k$.
}\end{proof}

Let $n\geq 3$ be an integer.  Using Theorem 8.6 and Theorem 8.4 in \cite{cartesian product} we see that 
$$\dim_M(And(1)\square C_n)=\dim_M(K_2 \square C_n)=\left\{ \begin{array}{ll}  2 & ~~$n$ ~is ~odd \\ 3 & ~~$n$~ is ~even \end{array} \right. $$
and
$$\dim_M(And(2)\square C_n)=\dim_M(C_5 \square C_n)=3.$$

\begin{pro}
If $k\geq 3$ and $n\geq 3$, then  $k\leq \dim_M(And(k)\square C_n)\leq k+1$.
\end{pro}
\begin{proof}{
Corollary 3.2 in \cite{cartesian product} using Theorem \ref{And(k)} implies that 
$$k=\max\{\dim_M(C_n),\dim_M(And(k))\}\leq dim_M(And(k)\square C_n).$$
For the upper bound, assume that 
$$V(C_n)=\{v_1,v_2,...,v_n\},~E(C_n)=\{v_1v_2,v_2v_3,...,v_{n-1}v_n,v_nv_1\}$$
and let
$$W'=\{(1,v_1),(4,v_1),(7,v_1),...,(3k-2,v_1)\},~W=W'\cup\{(1,v_2)\}.$$
Using the structure of the Cartesian product of two graphs, for each $i,j\in V(And(k))$ and for each $t,t'\in\{1,2,...,n\}$  we have 
$$d_{And(k) \square C_n} ((i,v_t),(j,v_{t'}))=d_{And(k)}(i,j)+\min\{ |t-t'|, n-|t-t'|\}.$$
Note that (using the proof of Theorem \ref{And(k)} and Lemma 3.1 in \cite{cartesian product}) the projection of $W'$ onto each copy of $And(k)$ in $And(k)\square C_n$ (i.e the induced subgraph on each row)  resolves the vertices of that copy and the projection of $W'$ onto each copy of $C_n$ in $And(k)\square C_n$ (each column) resolves its vertices. 
Also, for each $i$ and for each $1\leq t\leq \lfloor {n\over 2}\rfloor$ we have
$$ r\left((i,v_{1+t})|W'\right)=r\left((i,v_1)|W'\right)+(t,t,...,t)=r\left((i,v_{n-t+1})|W'\right). $$
Thus, distinct vertices in $\{ (i,v_t)|~i\in V(And(k)),~1\leq t\leq \lfloor {n\over 2}\rfloor +1\}$ have 
distinct metric representations with respect to $W'$ (and hence, with respect to $W$) and 
distinct vertices in 
$$\left\{ (i,v_t)|~i\in V(And(k)),~t\in\{ 1,n,n-1,n-2,..., \lceil {n\over 2}\rceil+1 \}\right\}$$
 have 
distinct metric representations with respect to $W'$ (and hence, with respect to $W$).
These facts imply that for each $i$ and for each $1\leq t\leq \lfloor {n\over 2}\rfloor$ the metric representation of the vertex $(i,v_{t+1})$  with respect to $W'$ in $And(k)\square C_n$ is just equal to the code of $(i,v_{n-t+1})$  and no other vertex
(note that when $n$ is even and $t={n\over 2}$ the vertex  $(i,v_{t+1})$ coincides with $(i,v_{n-t+1})$).
Since
$$d_{And(k)\square C_n}((i,v_{t+1}),(1,v_2))=d_{And(k)}(i,1)+t-1$$
 and 
$$d_{And(k)\square C_n}((i,v_{n-t+1}),(1,v_2))=d_{And(k)}(i,1)+\min\{ n-t-1, t+1 \}$$
for distinct vertices $(i,v_{t+1})$ and $(i,v_{n-t+1})$ we have $$r((i,v_{t+1})|W)\neq r((i,v_{n-t+1})|W).$$
Hence, $W$ is a resolving set  for $And(k)\square C_n$ and $\dim_M(And(k)\square C_n)\leq k+1$.
}\end{proof}

In the following theorems we investigate $\dim_M(And(k)\square K_n)$ and $\dim_M(Line(And(k) ) )$.



%


\end{document}